\documentclass{amsart}

\usepackage{amsfonts}
\usepackage{amsmath}
\usepackage{amssymb}
\usepackage{amsthm}
\usepackage[T1]{fontenc}
\usepackage[utf8]{inputenc}
\usepackage{latexsym}
\usepackage{longtable}
\usepackage{mathrsfs}
\usepackage{mathtools}
\usepackage{slashed}
\usepackage{stmaryrd}
\usepackage{tikz}
\usetikzlibrary{matrix,arrows}
\usepackage{xcolor}
\usepackage[all]{xy}
\usepackage{zhlipsum}
\usepackage{hyperref}
\hypersetup{
	colorlinks,
	citecolor=blue,
	filecolor=black,
	linkcolor=black,
	urlcolor=black
}

\newtheorem{theorem}{Theorem}[section]
\newtheorem{lemma}[theorem]{Lemma}
\newtheorem{proposition}[theorem]{Proposition}
\newtheorem{corollary}[theorem]{Corollary}

\theoremstyle{definition}
\newtheorem{definition}[theorem]{Definition}
\newtheorem{example}[theorem]{Example}

\theoremstyle{remark}
\newtheorem{remark}[theorem]{Remark}

\numberwithin{equation}{section}

\def  \Aut      {\mathrm{Aut}}

\def  \Char     {\mathrm{char}}

\def  \GL        {\mathrm{GL}}

\def  \Id       {\mathrm{Id}}

\def  \Stab     {\mathrm{Stab}}
\def  \supp     {\mathrm{supp}}

\def  \ca       {\mathcal{A}}

\def  \cC       {\mathcal{C}}

\def  \ch       {\mathcal{H}}

\def  \cl       {\mathcal{L}}

\def  \co       {\mathcal{O}}

\def  \bc       {\mathbb{C}}

\def  \bF       {\mathbb{F}}

\def  \bq       {\mathbb{Q}}
\def  \br       {\mathbb{R}}

\def  \bz       {\mathbb{Z}}

\usepackage{stmaryrd}

\begin{document}
\title{Hecke algebra of $\GL_n$ over a 2-dimensional local field}
\author{Xuecai Ma}
\address{ Westlake Institute for Advanced Study, Hangzhou, Zhejiang 310024, P.R. China \textcolor{white}{.} \\
	\hskip-.3cm\textcolor{white}{\raisebox{-.09cm}{$\bullet$}\hskip.04cm} Institute for Theoretical Sciences, Westlake University, Hangzhou, Zhejiang 310030, P.R. China
}

\curraddr{}
\email{maxuecai@westlake.edu.cn}
\thanks{}


\begin{abstract}
	Using the $\br((X))$-measure, we define and study  certain $\bc((X))$-valued functions on $\GL_n(F)$ for $F$ a two-dimensional local field. In particular, we define  a convolution product on  such suitable functions, which leads us to define the Hecke algebra of $\GL_n(F)$.   We  then define the measurable $\bc((X))$-representations of $\GL_n(F)$, and  prove that   function space  is  a candidate for such representations.
\end{abstract}

\maketitle

\section{Introduction}   
The concept of an $n$-dimensional field was first introduced by Ihara and Parshin in 1970s, which aims to generalizing the   classical adelic formalism to $n$-dimensional schemes.  
It is an important topic  to study representations of algebraic groups over $0$- and $1$-dimensional local fields. Since  finite fields and  local fields have good properties, for example, the  local fields are topological fields,  people know many things about their representations. But for algebraic groups over two-dimensional local fields, people know little about their representations. 

 In his paper \cite{kapranov2001double}, Kapranov studied the central extension $\Gamma$ of a reductive group $G$ over a two-dimensional local field $F=K((t))$, where $K$ is an 1-dimensional local field.  He chose an appropriate subgroup $\Delta_1 \subset \Gamma$, such that the fibres of the Hecke correspondences are locally compact spaces which can define invariant measures, so he can define the Hecke operators by integrating these measures.  He proved that such Hecke algebra $H(\Gamma, \Delta_1)$ is isomorphic to the double affine Hecke algebra associated to $G$.   After that,  Gaitsgory and Kazhdan gave a categorical framework of Kapranov's idea. They study representations in pro-vector spaces, see \cite{gaitsgory2004representations,  gaitsgory2005algebraic, gaitsgory2006algebraic} for more details and \cite{braverman2006some} for more examples.

 In  their paper \cite{braverman2011spherical},  Braverman and Kazhdan considered the subgroup $G^{+}_{\mathrm{aff}}$ of $G_{\mathrm{aff}}$, where $G_{\mathrm{aff}}$ is the semidirect product of $G_m$ and   the central  extension $\widetilde{G}$  of $G((t))$. They proved that any double cosets of $G_{\mathrm{aff}}(\co)$ inside $G^{+}_{\mathrm{aff}}$ is well defined and give rise to an algebra structure on a suitable space of $G_{\mathrm{aff}}(\co)$-biinvariant functions on $G^+_{\mathrm{aff}}(K)$. They call it the spherical Hecke algebra of $G_{\mathrm{aff}}$. And in \cite{braverman2016iwahori}, the authors used $W_X = W \ltimes X$, the semidirect product of the Weyl group with the Tits cone, to identify the double cosets $I \setminus   G^+_{\mathrm{aff}} /I$. By studying the combinatorics of the Tits cone $X$, they define the Iwahori-Hecke algebra of $G_{\mathrm{aff}}$.

On the other hand,  Fesenko constructed an $\br((X))$-measure on two-dimensional local fields \cite{fesenko2003analysis}, such that we can do harmonic analysis on two-dimensional local fields.  H. Kim and K. H. Lee use the Cartan decompositions to define the generators and relations of spherical Hecke algebra of $\mathrm{SL}_2$ \cite{kim2004spherical}. In particular, they proved the Satake isomorphism by using the $\br((X))$-measures.  One can also see \cite{lee2010iwahori} for the construction of Iwahori-Hecke Algebra for $\mathrm{SL}_2$ over a $2$-dimensional local field.

The present paper gives a new construction of Hecke algebra of $\GL_n(F)$ for $F$ a two-dimensional local field by using Fesenko's  $\br((x))$-measure and Morrow's work \cite{morrow2008integration} on integrations on $\GL_n(F)$. Specifically, we define an appropriate space of functions $f:  \GL_n(F) \to \bc((X))$ and prove that there is a well-defined convolution product in this space. We define and study a kind of representations called measurable representations. The typical example is the function space that we defined.

We now give the structure of this paper. In section 2, we first review  Fesenko's $\br((X))$-measure on two-dimensional local fields. We then introduce the integration on $\GL_n(F)$ for $F$ a two-dimensional local field. We define a space $\cl^M(\GL_n(F))$ consisting of functions $f: \GL_n(F) \to   \bc((X))$ satisfing certain conditions. We define the convolution product on $\cl^{H}(\GL_n(F))$, making $\cl^M(\GL_n(F))$  be an associative algebra, called the Hecke algebra of $\GL_n(F)$. In section 3, we define the measurable representations of $\GL_n(F)$, and prove the function space  $\cl^M{\GL_n(F)}$ is a measurable representation. This makes us establish a relation between the Hecke modules and measurable representations.

 \vspace{5pt}
\textbf{Acknowledgements.} The author thanks Ivan Fesenko for his support of this research and discussions about related topics of this work.  The author also thanks  Kye-Huan Lee for helpful discussions.
	This work was supported by Westlake University grant 207010146022301, 207010145609901 and Hangzhou Postdoctoral Daily Funding.

\section{Integration on $\GL_n(F)$}

  A $0$-dimensional local field is a finite field. For $n\geq 1$, an $n$-dimensional local field is a complete  discrete valuation field whose residue field is an $(n-1)$-dimensional local field.
 An one-dimensional local fields is one of $\br$, $\bc$, $\bF_{q}((t))$,  finite extensions of $\bq_p$.     Examples of two-dimensional local fields are $\bF_q((t_1))((t_2))$,  $K((t))$ for a local nonarchimedean field $K$,  $E((t))$ over a local archimedean field $E$,  finite extensions of $\bq_p \{\{t\}\}$. See \cite{fesenko2000invitation} for more introduction about higher local fields.

Let $F$ be an $n$-dimensional local field. Then there  is a sequence of integral rings
$$
F \supset  \co_F =\co_F^{(1)} \supset  \co_F^{(2)} \supset \cdots \supset \co_F^{(n)},
$$
where $\co_{F}^{(i)}$ is called the rank $i$ ring of integers of $F$.  Suppose that we have
a sequence of $n$-local  parameters $t_1, \cdots, t_n \in F$, then there are isomorphisms
$$
F= \co^{(n)}_F[t_1^{-1}, \cdots, t^{-1}_n], \text{ and } F^{\times} \cong (\co_F^{(n)})^{\times} \times t_1^{\bz} \times \cdots \times t_n^{\bz}.
$$

\begin{example}
  For the two-dimensional local field $\bq_p((t))$,  we have $\co^{(1)}_F =\bq_p[[t]] $ and $\co_F^{(2)}= \bz_p+t_2 \bq_p[[t_2]]$. For the local parameters, we can take $t_1=p$, $t_2=t$.
\end{example}

 One way to study local  fields  is using their topology. But  topological methods in higher dimensional local fields is not  so successful as in the one-dimensional cases.    Supposed that $F= K((t))$ is  an equal characteristic  two-dimensional local  field,   
Following \cite{parshin1984local},  we can define a topology  on $F$ by setting the basic neighbourhoods of $0$ are those of the form
	$$
	\sum_{i}U_i t^i := \left\{ \sum_{i}^{}a_i t^i \in F| a_i \in U_i  \text{ for all } i   \right\},
	$$
	where $\{U_i\}_{i= -\infty}^{\infty}$ are open neighbourhoods of $0$ such that $U_i=K$  for all $i  \gg  0$. The basic open neighbourhoods of any other point $a \in F$ are $a+U$, where $U_i$ are basic neighbourhoods of $0$.
  We know that one-dimensional local fields are topological fields. But $K((t))$ is not a topological field, since the multiplication and inverse are not continuous. Moreover,  $K((t))$  is not a locally compact space,  so there is no Harr measure on it.  For the topology of mixed characteristic  two-dimensional local fields,  it is defined in \cite{fesenko2001sequential}.

\subsection{Measures   and  integrations  of two-dimensional local fields}
In the following contents,  we  let  $F$ be a two-dimensional local field with first residue field $E$ and second residue field $F_q$, rank one ring of  integers $\co$, rank two ring of integers $O$,  local parameters  $t_1, t_2$.

Let $\ca$ be the minimal ring generated by the distinguished sets $\{ \alpha+ t_2^it_1^jO  \}$.  Elements of $\ca$ can be written as a finite disjoint union sets $A_n$,  where each $A_n$ is a different between a distinguished set and a finite union of distinguished sets.                                                                          
\begin{proposition}    \textnormal{(\cite{fesenko2003analysis})}
	There is a unique measure  $\mu$ on $F$  with values in $\mathbb{R}((X))$,  which is a translation invariant  and finitely addictive  such that
	$$
	\mu ( t_2^i   t_1 ^j  O)  = q^{-j} X^i.
	$$
\end{proposition}
\begin{example}
We have $\mu(O)=1$ and $\mu(t_2^i p^{-1} (S)) =X^i \mu_E(S)$,  where $\mu_E$ is the normalized Haar measure on $E$  such that $\mu_E(\co_{E})=1$ and $S$ is a compact open subsets of $E$.
\end{example}

We call a sum 
$$
\underset{n}{\sum} \underset{i}{\sum} a_{i,n} X^i
$$
is absolutely convergent  if
\begin{enumerate}
	\item  there is  an $i_0$ such that $a_{i,n} =0$ for all $ i\leq i_0$ and all $n$;
	\item  for every $i$ the series $\sum_{n} a_{i,n}$ absolutely convergent in $\bc$. 
\end{enumerate}

\begin{remark}
	The measure on $F$ is additive in the following sense: suppose that $\{A_n\}$ are countably many disjoint sets of $\ca$ such that $\bigcup A_n  \in \ca$ and $\sum \mu(A_n) $ absolutely converges in $\bc((X))$, then we have $\mu(\bigcup A_n)= \sum \mu(A_n)$.
\end{remark}
Having a measure of $F$, we can define the integrations.   We consider functions $f: F \to \bc((X))$.
\begin{enumerate}
	\item  For any $A \in \ca$, we  consider the  function $\mathrm{char}_{A}$ and we define
	$$
	\int_F   \Char_{A}  = \mu(A).
	$$
	\item  For functions $f$ can be written as $\sum c_n\Char_ {A_n}  +   \sum  a_i  \Char_{\{p_i\}}$, where  $\{A_n \}$ are countably disjoint measurable sets in $\ca$, $c_n \in \bc((X))$ such that $\sum c_n \mu(A_n)$ is absolutely converges in $\bc((X))$, and $\{p_1, \cdots, p_m\}$ is  a collection of  finitely many points. We define
	$$
	\int_F ( \sum c_n\Char_ {A_n}  +   \sum_{i=1}^{k}  a_i  \Char_{\{p_i\}}  ) d \mu = \sum c_n \mu(A_n).
	$$
\end{enumerate}

Since $\ca$ contains $a+t_2^jp^{-1}(S)$, for $S$ a compact open subsets in $E$. We have
$$
\int_F  \Char_{a+t_2^jp^{-1}(S)}  d  \mu = \mu_E(S)  X^j.
$$
If we suppose the coefficients $c_k \in  \bc$.
 Then  
 $$
 \int_F  \sum_k   c_k   \Char_{a+  t_2^i t_1^{j_k} O}   d \mu =  \left(\sum c_k  q^{-j_k}  \right) X^i    = \left( \int_E \sum c_k \Char^{E}_{t_1^{j_k} O_E }  d \mu_E \right)X^i=  \left(\int_E f d\mu_E\right) X^i,
  $$
for some integrable $f$  on $E$. The integrations on $F$ defined above can also be
defined by the following procedure,  see  \cite{morrow2010integration} for more details.
\begin{enumerate}
	\item  First, we let $\cl(E)$ denote the space of locally constant integrable functions  on $E$. 
	\item   For $g \in \cl(E)$, $a \in F,   i \in \bz$,   we  define a  function on $F$ by
	$$
	 g^{a, i }(x) = \left \{
	 \begin{array}{ll}
	 g( \overline{(x-a)t_2^{-i}} )  &  x \in a + t_2^i \co  \\
	 0    &  \text{otherwise}.   \\
	 \end{array} 
	 \right.
	$$
	\item   A simple function on $F$ is a $\bc((X))$-valued function of the form
	$$
	x \mapsto g^{a,i }(x) X^{j} 
	$$
	for some $g \in \cl(E)$, $a \in F$, $ i, j \in \bz$.
	\item  Let $\cl(F)$ denote the space of all $\bc((X))$-valued functions spanned by simple functions. 
	\item   For simple functions,  we define
	$$
	\int_F g^{a,i}(x)  dx=  \left(\int_ E g(u) du \right) X^i,
	$$
and extend it to all functions in $\cl(F)$.
\end{enumerate}

\subsection{Integrations on  the product space $F^n$}
We first review the approach of  \cite{morrow2008integration}.   We first consider  integrations on $F^n$.   The important thing of repeated integral is  the Fubini property, says that for a function
$$
f: F^n \to \bc((X)),
$$
and  for each permutation  $\sigma $  of $\{1, \cdots, n\}$, the expression
$$
\int^F  \cdots \int^F f(x_1,  \cdots ,x_n) dx_{\sigma(1)} \cdot dx_{\sigma(n)}
$$
is well defined and its value doesn't depend on $\sigma$.
\begin{enumerate}
	\item  Let $\cl(E^n)$ denote the space of functions $f: E^n \to \bc$ which are Fubini.
	\item For every $g \in \cl(E^n)$, and every $a=(a_1, \dots,  a_n) \in F^n$, $\gamma =(\gamma_1, \dots, \gamma_n) \in  \bz^n$. The  product of translated fraction ideals is given by
	$$
	 a+ t_2^{\gamma} \co_F^n=    \prod  a_i + t_2^{\gamma_i}\co_F    \quad   \in F^n.
	$$
	We define a function on $\GL_n(F)$
		$$
	g^{a, \gamma}(x) = \left \{
	\begin{array}{ll}
		g(\overline{(x-a)t_2^{-\gamma}} )  &  x \in a + t_2^{\gamma} \co_F^n  \\
		0    &  \text{otherwise}.   \\
	\end{array} 
	\right.
	$$
	\item  For a function $g^{a, \gamma}$ defined above, we define a integration 
	$$
	\int_{F^n}g ^{a, \gamma}(x
	)dx  = \left(\int_{E^n}  g d\mu_{E^n}  \right) X^{\sum_{i=1}^{n} \gamma_i}.
	$$
	 \item A complex function $f: E^N \to \bc$ is called $GL$-Fubini if and only $f  \circ \tau$ is Fubini for all $\tau \in \GL_n(F)$. Let $\cl(F^n, \GL_n)$ denote the $\bc((X))$ space-valued functions $\GL_n(F) \to \bc((X))$ spanned by
	 $$
	 g^{a ,\gamma}\circ \tau, \quad  \tau \in \GL_n(F), a \in F^n, \gamma \in \bz^n.
	 $$
\end{enumerate}

\begin{theorem}(\cite[Theorem 3.4]{morrow2008integration})  \label{fnfubini}
	 Every function in $\cl(F^n, \GL_n)$ is Fubini on $F^n$. If $f \in \cl(F^n ,GL_n), a \in F^n, \tau \in \GL_n(F)$, then the functions $x \mapsto f(x+a)$ and $x \mapsto f(\tau x)$ also belong to $\cl(F^n, \GL_n)$, and the repeated integrals are given by
	 $$
	 \int_{F^n}  f(x +a) dx = \int_{F^n} f(x)dx,
	 $$
	 $$
	   \int_{F^n}f(\tau  x) dx = |\det \tau|^{-1} \int_{F^n} f(x)dx.
	 $$
\end{theorem}

\vspace{1em}
\subsection{Integrations on $\GL_n(F)$}

We now consider  integrations on $\GL_n(F)$.  Let $T:    F^{n^2}\to M_n(F)$ be the isomorphism from the vector space $F^{n^2}$   to the vector space of $n \times n$ matrices with items in $F$.  Let $\cl(M_n(F))$ be the space of $\bc((X))$-valued functions $f$ such that $f \circ T$ belongs to $\cl(F^{n^2}, \GL_n)$, and we define
$$
\int_{M_n(F)} f(x)dx = \int_{F^{n^2}}  f \circ T(x) dx.
$$
Let $\cl(G_n(F))$ denote the space of $\bc((X))$-valued functions $\phi$ on $\GL_n(F)$  such that  $\tau \mapsto  \phi(\tau)|\det \tau|^{-n} $  extends to a function of $\cl(M_n(F))$. The integral of $\phi$ over $\GL_n (F)$   is defined by
$$
\int_{GL_n(F)} \phi(\tau) d\tau  = \int_{M_n(F)} \phi(x)  |\det(x)|^{-n} dx.
$$

\begin{proposition}
	The integrals defined above have  the following properties:
\begin{enumerate}
\item The integrals on $\GL_n(F)$ is well defined, which means that if $f_1, f_2 \in \cl(M_n(F))$  are equal when restricted to $\GL_n(F)$, then $f_1= f_2$.   

\item Suppose that $\phi \in \cl(\GL_n(F))$ and $\sigma \in \GL_n(F)$, then $\tau \mapsto  \phi( \tau \sigma)$ and $\tau \mapsto  \phi(\tau \sigma)$ also belong to $\cl(GL_n(F))$, with
$$
\int_{\GL_n(F)} \phi(\sigma \tau)  d\tau =  \int_{\GL_n(F)}  f(\tau)  d \tau  =  \int_{\GL_n(F)}  \phi(\tau \sigma)  d \tau.
$$

\item Suppose   that $g$ is a complex valued Schwartz-Bruhat function on $\GL_n(E)$ such that 
$$
f(x)= \left\{
\begin{array}{ll}
	g(x)|\det x|^{-n}  &  x \in \GL_n(E)  \\
	0     &   \det x =0  \\
\end{array}
\right.
$$
is $GL$-Fubini on $M_n(E)$.  We define $g^0 : \GL_n(F) \to \bc((X))$ by 
$$
g^0(x) = \left\{
\begin{array}{ll}
	g(\overline{x})  &   x   \in  \GL_n(\co).  \\
	0   &  \text{otherwise}. \\  
\end{array}
\right.
$$Then $g^0$ belongs to $\cl(GL_n(F))$, and
$$
\int_{\GL_n(F)}  g^{0}  (\tau)  d\tau  = \int_{\GL_n(E)} g(u) du.
$$
\end{enumerate}

\end{proposition}
\begin{proof}
	These  follows from\cite[Remark 4.3, Proposition 4.4, Proposition 4.8]{morrow2008integration}.                                                                                                                                                                                                                                                                                                                                                                                                                                                                                                                                                                                                                                                                                                                                                                                                                                                                                                                                                                                                                                                                                                                                                                                                                              
\end{proof}

\begin{remark}
	Using these definitions of \cite{morrow2008integration}, we find that
	$$
\sum c_k \prod_{i=1}^n  \mu (a^k_i+  t_2^{\gamma^k_i}t_1^{\delta^k_i}   O^n )  =  \sum c_k\prod_{i=1}^n   t_1^{-\delta^k_i} X^{\gamma^k_i}=  \sum c_k \left ((\prod  \mu_E(t_1^{\delta^k_i}  O_E))  X^{\sum \gamma^k_i} \right).
	$$
We then have
	$$
	\sum c_k \prod_{i=1}^n  \mu (a^k_i+  t_2^{\gamma^k_i}t_1^{e^k_i}   O^n ) = \sum c_k \left( (\int_{E^n}  f_k d \mu_{E^n} )   X^{\sum_i \gamma^k_i}   \right).
	$$
Like the one-dimensional case,  a distinguished set of $F^n$ is of the form
 	$$
    \prod_{i=1}^n  a_i  +  t_2^{\gamma_i}t_1^{\delta_i} O,   \quad   a_i \in F , \gamma_i, \delta_i \in  \bz,
 	$$
 	and its measure is $\mu(  \prod_{i=1}^n  a_i  +  t_2^{\gamma_i}t_1^{\delta_i} O)  = \prod_{i=1}^n \mu( a_i  +  t_2^{\gamma_i}t_1^{\delta_i} O) = \prod_{i=1}^n q^{-\delta_i} X^{\gamma_i}$.   Let $\ca^{\mathrm{dist}}_{F^n}$ be the collection of distinguished sets and let $\ca_{F^n}$ be the minimal ring generated by the distinguished set. Then for finitely  many disjoint  subsets $A_k \in \ca_{F^n}$, we have $\sum  c_k A_k  = \sum c_k \mu (A_k)$.
\end{remark}

\begin{remark}
	 It is not an easy  thing to define distinguished sets of $\GL_n(F)$, since usually the measure on $\GL_n (F)$ should be  $ |\det x|^{-n}  d_{M_n(F)} x$, but the determinant is not usually a constant.  In some special cases, there  are some results,  for example, see \cite{waller2019measure} for the description of measurable subsets of $\GL_2(F)$.
\end{remark}

\begin{lemma}\label{basicintersection}
	   For any $a, b \in F, \gamma, \delta \in \bz$,  we consider   two subsets $a+t_2^{\gamma} \co$ and $b+ t_2^{\delta} \co $.  Then the intersection  subset of them  is either empty or also has the form $c + t_2^{\epsilon} \co$.
\end{lemma}

\begin{proof}
	 We may suppose that $ a \notin t_2^{\gamma} \co$ and $b \notin t_2^{\delta} \co$, otherwise let $a$ or $b$ equals $0$.    We have the following cases:
	 \begin{enumerate}
	 	\item  $\gamma = \delta$.  If  $a =b$ , then $(a+t_2^{\gamma} \co) \cap  (b+t_2^{\delta} \co)    =a+t_2^{\gamma} \co =b+t_2^{\delta} \co$; If $a \neq b$, then $a+t_2^{\gamma} \co  \cap a+t_2^{\gamma} \co= \emptyset$.  
	 	\item  $\gamma  >  \delta$.    If  $a =b$ , then $(a+t_2^{\gamma} \co) \cap  (b+t_2^{\delta} \co)    =a+t_2^{\gamma} \co$; If $a \neq b$, then $(a+t_2^{\gamma} \co)  \cap( a+t_2^{\gamma} \co )= \emptyset$. 
	 	\item  $\gamma  <  \delta$.    If  $a =b$ , then $(a+t_2^{\gamma} \co) \cap  (b+t_2^{\delta} \co)    =b+t_2^{\delta} \co^n $; If $a \neq b$, then $(a+t_2^{\gamma} \co)  \cap( a+t_2^{\gamma} \co)= \emptyset$.
	 \end{enumerate}
	  Then we get the lemma.
\end{proof}
\begin{lemma}\label{linearinvariant}
	Let $\tau \in \GL_n(F)$. Then for any  set  of the form $A +t_2^{\Gamma}\co^n =\prod_{1 \leq i \leq n} (a_i + t_2^{\gamma_{i}} \co)$, $\tau (\prod_{1 \leq i \leq n} (a_i + t_2^{\gamma_{i}} \co))$ also has this form.
\end{lemma}

\begin{proof}
	This  lemma	 can be proved by the following two results:
	\begin{enumerate}
		\item  For any  $g \in F$ and  $a+ t_2^{\gamma} \co$,  $g(a+ t_2^{\gamma}\co)$ also has the form $a'+ t_2^{\gamma'}\co$.  To prove this,	 we may suppose that $ g =\sum^{\infty}_{i=i_0} a_i t_2 ^i, a_{i_0}\in E^*$, then $g(a+ t_2^{\gamma} \co)= ga +t_2^{i_0+\gamma}\co$, this is because $(a_{i_0}+ \sum_{i=i_0+1} a_i t_2^{i-i_0})$ is a unit in $\co= E((t))$.

		\item For any  two subsets $a+ t_2^{\gamma} \co$ and $b+ t_2^{\delta} \co$, their sum  $(a+ t_2^{\gamma} \co)+(b+ t_2^{\delta} \co)$ also has the form $c+ t_2^{\epsilon} \co$.  We have the following cases:
		\begin{enumerate}
			\item If $\gamma= \delta$, then $(a+ t_2^{\gamma} \co)+(b+ t_2^{\delta} \co)= a+b+ t_2^{\gamma} \co = a+b+ t_2^{\delta}\co$.
			\item If $\gamma>\delta$, then $(a+ t_2^{\gamma} \co)+(b+ t_2^{\delta} \co) = (a+ b+ t_2^{\delta} \co)$.
			\item f $\gamma< \delta$, then $(a+ t_2^{\gamma} \co)+(b+ t_2^{\delta} \co) = (a+ b+ t_2^{\gamma} \co)$.
		\end{enumerate}
	\end{enumerate}
\end{proof}
\begin{lemma}\label{intersection}
	For any two subsets in $\GL_n(F)$ with the form $ \tau(\prod_{1 \leq i \leq n} (a_i + t_2^{\gamma_{i}} \co))$,  $\tau \in \GL_n(F)$, their intersection is either empty or also has the form $ \prod c_{i} + t_2^{\epsilon_{i}} \co $.
\end{lemma}

\begin{proof}
	We first assume  $\tau = I_n$. For two subsets of the form $\prod_{1\leq i\leq n}(a_i+ t_2^{\gamma_{i} }\co)  $  and $ \prod _{1 \leq i \leq n}(b_{i} + t_2^{\delta_{i}} \co)$, their intersection is
	\begin{enumerate}
		\item empty, if there is an $i$ such that $a_{i}+ t_2^{\gamma_{i}} \cap  b_{i}+ t_2^{\delta_{i} }  =\emptyset$;
		\item    $\prod_{1 \leq i\leq n} c_{i} + t_2^{\epsilon_{i}} \co $, where  $(c_{i} + t_2^{\epsilon_{i}} \co)$  comes from Lemma \ref{basicintersection}.
	\end{enumerate}
	Thus the lemma is true for $\tau$ is $I_n$.
	For general $\tau_a (\prod_{1 \leq i \leq n} (a_i + t_2^{\gamma_{i}} \co)  )$  and  $\tau_b(\prod_{1 \leq i \leq n} (b_i + t_2^{\delta_{i}} \co) $, they  also  can be written as the form  $\prod_{1 \leq i \leq n} (d_i + t_2^{\zeta_{i}} \co) $. By the  $\tau= I_n $ case,  we get the desired conclusion.
\end{proof}

\begin{lemma}\label{product}
	For $f_1 , f_2 \in \cl(F^n, \GL_n(F))$, we have $f_1\cdot  f_2 \in  \cl(F^n,  \GL_n(F))$.
\end{lemma}
\begin{proof}
 By Lemma \ref{linearinvariant} and linear properties, we may suppose that $f_1 = g_1^{A_1 , \Gamma_1}, f_2 = g_2^{A_2, \Gamma_2} $. By  Lemma \ref{intersection}, we have  either $\supp f_1 \cap   \supp f_2  = \emptyset $  or $\supp f_1 \cap \supp f_2 =A+ t_2^{\Gamma} \co^n= \prod_{1 \leq i \leq n} a_i + t_2^{\gamma_{i}} \co$. If $\supp f_1 \cap   \supp f_2 = \emptyset$, then $f_1 .f_2  =0$.     If $\supp f_1 \cap \supp f_2 = A+ t_2^{\Gamma} \co^n$, we have 
	\begin{enumerate}
		\item  $g_1 ^{A_1, \Gamma_1}|_{\supp f_1 \cap \supp f_2}   = h_{1}^{A, \Gamma}$, where $h_1(x) = g_1(x+  \overline{(A_1 -A) t_2^{-\Gamma}}) $  and
		\item $g_2 ^{A_2, \Gamma_2}|_{\supp f_1 \cap \supp f_2}   = h_{2}^{A, \Gamma}$, where $h_2(x) = g_1(x+  \overline{(A_2 -A) t_2^{-\Gamma}}) $.
	\end{enumerate}
	We then get $f_1\cdot f_2 = (h_1 \cdot h_2)^{A, \Gamma}$.
\end{proof}

\begin{corollary}  \label{productclosed}
	For $f_1,  f_2  \in \cl(M_n(F))$, we have $f_1 \cdot f_2  \in \cl(M_n(F))$.    
\end{corollary}

\vspace{1em}

\subsection{Convolution product of integrable functions over $\GL_n(F)$}

For any $g : \GL_n(E)  \to \bc $,  $A \in  \GL_n(F), \Gamma = (\Gamma_{i,j})_{1 \leq i,j  \leq n} \in M_n(\bz_{+})$. Let $t_2^{\Gamma} M_n(\co)$ denote set of matrices $(t_2^{\Gamma_{i,j}} x_{i,j })_{1\leq i, j \leq n }, x_{ij} \in \co$, we define a function $g^{A, \Gamma}: \GL_n(F) \to \bc((X))$ by 
$$
g^{A, \Gamma}(x) = \left\{
\begin{array}{ll}
	g\left(\overline{(A^{-1}x-I_n) t_2^{-\Gamma}}\right)  &   x   \in  A(I_n+t_2^{\Gamma} M_n(\co) ), \\
	0   &  \text{otherwise}.  \\  
\end{array}
\right.
$$
 And we define 
 $$
 g^{0, 0}(x) = \left\{
 \begin{array}{ll}
 	g(\overline{x}) &   x   \in  \GL_n(\co), \\
 	0   &  \text{otherwise}.  \\  
 \end{array}
 \right.
 $$

\begin{definition}
	Let $V$  be a subset of $\GL_n(F)$,  we  will say  that $V$ is a measurable subset if $\Char_V$
	belongs to $\cl(\GL_n(F))$, i.e., $ \Char_V |\det x|^{-n}$ extends to a function in $\cl(M_n(F))$.
\end{definition}

\begin{lemma}
Let $g$ be a function on $\GL_n(E)$ such that  
$$
f(x)=\left\{ 
\begin{array}{ll}
	g(x) |\det x|^{-n}  &   x \in \GL_n(E) \\
	0          &    \det x  =0  \\
\end{array}
\right.
$$
is  $GL$-Fubini on $M_n(E)$. For any $A \in  \GL_n(F), \Gamma \in M_n(\bz_{+})$ such that  $A(I_n+t_2^{\Gamma} M_n(\co))$ is  a measurable subset, we have $g^{A, \Gamma}    \in \cl(\GL_n(F))$.
\end{lemma}
\begin{proof}
	Since $\int_{\GL_n(F)}$ is translation invariant, it is sufficient to prove the proposition for $g^{0,\Gamma}$.
	We have
	$$
	g^{0, \Gamma}(x) = \left\{
	\begin{array}{ll}
		g(\overline{(x-I_n) t_2^{-\Gamma}})  &   x   \in  I_n+ t_2^{\Gamma} M_n(\co),  \\
		0   &  \text{otherwise}. \\  
	\end{array}
	\right.
	$$
	To prove that it belongs to $\cl(\GL_n(F))$,  we must have  $g^{A ,\Gamma}(x) |\det x|^{-n}$ extends to a function in $\cl(M_n(F))$. We already have $g^{0, \Gamma}$ belongs to $\cl(M_n(F))$, and $\Char_{\supp_{g^{0, \Gamma }}  } |\det x|^{-n}$ also belongs to $\cl(M_n(F))$.  Then the lemma follows from Corollary \ref{productclosed}.
	
\end{proof}

\begin{definition}
	We let $\cl^M(\GL_n F)$  denote  the  the  $\bc((X))$-subspace of $\cl(\GL_n(F))$ consists of  functions $f: \GL_n (F) \to \bc((X))$ generated by $g^{0,0}$ and
	$$
	g^{A, \Gamma},   \quad  A  \in \GL_n(F),   \Gamma = (\Gamma_{i,j})_{1\leq i,j  \leq n} \in M_n(\bz_+),
	$$	 
	which  satisfies 
		\begin{enumerate}

	 \item   $g$ is a complex valued Bruhat-Schwartz function on $\GL_n(E)$,
		\item  $
		f(x)=\left\{ 
		\begin{array}{ll}
			g(x) |\det x|^{-n}  &   x \in \GL_n(E) \\
			0          &    \det x  =0  \\
		\end{array}
		\right.
		$
		is  $GL$-Fubini on $M_n(E)$ , and
    \item   the support  of $g^{A, \Gamma}$ is a measurable subset.
	\end{enumerate}
	
\end{definition}

\begin{remark}\label{productsupport}
 Here the superscript M means measurable. Functions of the form $g^{A, \Gamma}, A \in \GL_n(F), \Gamma \in M_n(\bz_+)$   and satisfies  the above conditions are called basic functions in $\cl^{M}(\GL_n(F))$. It is clear that $\cl^{M}(\GL_n(F))$ is a  subspace of $\cl(\GL_n(F))$.   Since Bruhat-schwartz functions are locally constant functions,   we get $\cl^{M}(\GL_n(F))$ is generated by $\Char_V$, $V$ are  measurable subsets of $\GL_n(F)$.  Let $V_1, V_2$ be two measurable subsets of $\GL_n(F)$, we  have $\Char_{V_1} |\det x|^{-n},  \Char_{V_2}|\det x|^{-n}  \in \cl (M_n(F))$.  By the definition of $\cl(M_n(F))$, we have $\Char_{V_1}|\det x|^{-n}$  have the form $\sum a_i  \Char_{U_i}$, $U_i$ are disjoint measurable subsets of $ M_n(F)$,  $\Char_{V_2}|\det x|^{-n}$  have the form $\sum b_i  \Char_{W_i}$, $W_i$ are disjoint measurable subsets of $ M_n(F)$.   We also notice that $\Char_{V_2} = \Char_{\cup W_i}   \in \cl(M_n(F))$. We then have
 $$
 \Char_{V_1 \cap V_2} |\det x|^{-n}= \Char_{V_1}  \Char_{V_2}  |\det x|^{-n}=  (\sum a_i \Char_{U_i})\Char_{ \cup W_i}   \in \cl(M_n(F)).
 $$
By Lemma \ref{productclosed},  we get the intersection of two measurable subsets  in $\GL_n(F)$ is still a measurable subset.
\end{remark}

\begin{definition}
	For every $f_1,f_2  \in  \cl^M (\GL_n(F))$, we define the convolution product of $f_1$ and $f_2$ to be
	$$
	f_1* f_2(y):= \int_{\GL_n(F)} f_1(yg^{-1}) f_2(g) dg.
	$$
\end{definition}

\begin{lemma}\label{productclosedgln}
	  For any two basic functions $g_1^{A, \Gamma}$ and $g_2^{B, \Delta}$, the product
	  \[
	  g_1^{A, \Gamma} \cdot g_2^{B,  \Delta} 
	  \]
	  is also a basic function.
\end{lemma}

\begin{proof}
	  We have $\supp g_1^{A, \Gamma}= A(I_n+  t_2^{\Gamma} M_n(\co))\subset \GL_n(F) \subset M_n(F)$ and $\supp g_2^{B, \Delta}= B(I_n+  t_2^{\Delta} M_n(\co)) \subset \GL_n(F) \subset M_n(F)$.  Applying the isomorphism
	  $$
	  T^{-1}:  M_n(F)  \to F^{n^2},
	  $$
	   we find that $ T^{-1} (A(I_n+  t_2^{\Gamma} M_n(\co)))$  has the form $\tau_A(   T^{-1}(I_n)  +  T^{-1} (t_2^{\Gamma}) \co^{n^2})\subset F^{n^2}$, similarly for $B(I_n+  t_2^{\Delta} M_n(\co))$. We can view $g_1^{A, \Gamma}$ and $g_2^{B, \Delta}$ as functions on $M_n(F)$,  By the intersection properties of these sets (Lemma \ref{intersection}) and applying the functor  $T$, we then  get  $ \supp g_1^{A, \Gamma} \cap \supp g_2^{B, \Delta} = C+ t_2^{\Upsilon} M_n(\co)$.  Moreover  by Lemma \ref{product}, we get $g_1^{A, \Gamma} \cdot g_2^{B, \Delta} =h'^{ C,  \Upsilon} $,  where  $h'$ is a complex valued Bruhat-Schwartz function on $\GL_n(E)$, such that
	  $$
	   x \mapsto \left\{ 
	   \begin{array}{ll}
	   	h'(x) |\det x|^{-n}  &   x \in \GL_n(E) \\
	   	0          &    \det x  =0  \\
	   \end{array}
	   \right.
	   $$
	   is  $GL$-Fubini on $M_n(E)$,  and $h'^{ C,  \Upsilon} $ is defined by
	   $$
	   h'^{C, \Upsilon}(x) = \left\{
	   \begin{array}{ll}
	   	g\left(\overline{(x-C) t_2^{-\Upsilon}}\right)  &   x   \in  C+t_2^{\Upsilon} M_n(\co) , \\
	   	0   &  \text{otherwise}.  \\  
	   \end{array}
	   \right.
	   $$
	    We notice that this is different from our general  lifting functions $g^{A, \Gamma}: GL_n(F) \to \cC((X))$ whose supports are $A(I_n +t_2 ^{\Gamma}M_n(\co))$.   So we only need to prove that $C+t_2^{\Upsilon} M_n(\co)= C(I_n +t_2^{\Upsilon}C^{-1} M_n(\co) )$ can be written as $D(I_n +t_2^{\Theta}M_n(\co))$. This is easy, we just  notice that $t_2^{\Gamma}A,  t_2^{\Delta}B \in \GL_n(\co)$ by our condition,  and $C+t_2^{\Upsilon} M_n(\co) \subset A(I_n+t_2^{\Gamma} M_n(\co)) \cap B(I_n+t_2^{\Delta} M_n(\co))$,
	    hence $t_2^{\Upsilon}C^{-1} \in \GL_n(\co)$.
	    
	    Clearly, the product is a lift of Bruhat-Schwrtz function, and the support of the product is still a measurable subset by Remark \ref{productsupport}.
\end{proof}

\begin{lemma}
	 For any two basic functions $g_1^{A, \Gamma}, g_2^{B, \Delta}$, and  a fixed $y \in \GL_n(F)$, the product
	 $$
	   g_1^{A, \Gamma}(y
	    x^{-1}) \cdot g_2^{B,  \Delta}(x) 
	 $$
	  is also a basic function with variable $x$.   And the function $g_1^{A, \Gamma}(yx^{-1})$ can be written as $h_1^{A, \Gamma}(y^{-1}x)$.
\end{lemma}
\begin{proof}
	 Since Bruhat-Schwartz functions are locally constant functions with compact supports.  We  may suppose that $g_1 =\Char_{V_1} , g_2= \Char_{V_2}: \GL_n(E) \to \bc$, $V_1, V_2$ are compact subsets of $\GL_n(E)$.   Let $p: \GL_n( \co) \to \GL_n(E)$  be the projection map. We then have
	 $$
	 g_1^{A, \Gamma} (yx^{-1})= \Char_{U_1}, U_1= \left\{x \in \GL_n(F)| x^{-1} \in  y^{-1}A(I_n+ t_2^{\Gamma}p^{-1} (V_1) \right\}
	 $$
    $$
    g_2^{B,  \Delta}(x) = \Char_{U_2}, U_2= \left\{x \in \GL_n(F)| x  \in  B(I_n+ t_2^{\Delta}p^{-1} (V_2)) \right\}.
    $$
     For $x \in \GL_n(F)$ such that $ x^{-1} =y^{-1}A(I_n +t_2^{\Gamma}M )$,   where $M \in M_n(\co),  p(M) \in V_1$.  We get $x = yA^{-1} (I_n - t_2^{\Gamma} M + (t_2^{\Gamma} M )^2 + \cdots ) = yA^{-1}(I_n +t_2^{\Gamma}( -M + t_2^{\Gamma}M +\cdots )$.  We notice that  $\Gamma_{i,j} > 0$, then $p(-M+ t_2^{\Gamma}M+ (t_2^{\Gamma}M)^2 =\cdots ) = p(-M)$.  Therefore $x \in  yA^{-1}(I_n +t_2^{\Gamma}p^{-1}(-V_1))$. Thus $U_1 \subset yA^{-1}(I_n +t_2^{\Gamma}p^{-1}(-V_1))$ Moreover, for any $z \in  yA^{-1}(I_n +t_2^{\Gamma}p^{-1}(-V_1))$, we have $z^{-1} \in  y^{-1}A(I_n +t_2^{\Gamma}p^{-1}(-V_1))$ by the same discussion.  Hence $y A^{-1}(I_n +t_2^{\Gamma}p^{-1}(-V_1)) \subset U_1$ since $z= (z^{-1})^{-1}$. We then have $g_1^{A, \Gamma}(y x^{-1}) =  (\Char_{V_1})^{A, \Gamma} = (\Char_{-V_1})^{A^{-1}, \Gamma}(y^{-1}x) =h_1^{A^{-1}, \Gamma}(y^{-1}x)$.
     
      Therefore $g_1^{A, \Gamma}(yx^{-1})$ is also a basic function with variable $x$. Then the lemma follows from Lemma \ref{productclosedgln}.
     
\end{proof}

\begin{proposition}
	For any two functions $f_1, f_2 \in   \cl^M(\GL_n(F)$, $f_1*  f_2  \in \cl^M(\GL_n(F))$.
\end{proposition}

\begin{proof}
 We may assume $f_1 = g_1^{A_1 , \Gamma_1}  , f_2 = g_2^{A_2, \Gamma_2}$.
  then we  have $	g_1^{A_1 , \Gamma_1}  (y x^{-1}) g_2^{A_2 , \Gamma_2}  (x) $ belong to $\cl^M(\GL_n(F))$, and $\Char_{f_2}|\det x|^{-n}$  also belongs to  $\cl^{M}(\GL_n(F))$.   We  compute the convolution product:
  \begin{eqnarray*}
  	f_1* f_2(y):&=& \int_{\GL_n(F)} f_1(yx^{-1}) f_2(x) dx=   \int_{M_n(F)}   f_1(yx^{-1})  f_2(x) |\det x|^{-n}  dx   \\
             &=&   \int_{ M_n(F)} g_1^{A_1, \Gamma_1}  (yx^{-1}) g_2^{A_2, \Gamma} (x) \left( \sum a_w g_w^{A_w, \Gamma_w}(x) \right)  dx  \\ 
         &=&   \sum a_w \int_{ M_n(F)} g_1^{A_1, \Gamma_1}(yx^{-1} ){g'_{w}}^{A'_w, \Gamma'_w}  ( x) dx,  \\ 
 \end{eqnarray*}
 where the last equality comes form    the product of two  basic functions is still a basic function.
  To make  the integration $ \int_{ M_n(F)}  g_1^{A_1, \Gamma_1}(yx^{-1} ){g'_{w}}^{A'_w, \Gamma'_k}  ( x)$ not equals $0$ , we must have $ y \in  A_1(I_n+ t_2^{\Gamma_1}) A'_w(I_n+ t_2^{\Gamma'_w} )= C_w(I_n+t_2 ^{\Delta_w}M_n(\co))$. 

 \[
  	f_1* f_2(y) = \sum a_w  \Char_{ C_w(I_n+t_2 ^{\Delta_w}M_n(\co))}\int_{M_n(F)}  \left(   g_1^{A_1, \Gamma_1}(yx^{-1}))  (  {g'_{w}}^{A'_w, \Gamma'_w} (x)) \right)   dx .
\]
 \[
= \sum a_w  \Char_{ C_w(I_n+t_2 ^{\Delta_w}M_n(\co))}\int_{M_n(F)}  \left(   h_1^{A^{-1}_1, \Gamma_1}(y^{-1}x))  (  {g'_{w}}^{A'_w, \Gamma'_w} (x)) \right)   dx .
\]
    In the following, we assume that  $y \in  C_w(I_n+t_2 ^{\Delta_w}M_n(\co))$.  It is easy to  find  that   for a simple function $g^{A,\Gamma}$ on $\GL_n(F)$, $g^{A, \Gamma} \circ T(x)$  is  a simple function  $g^{a, \gamma}: F^{n^2} \to \bc((x))$ on $F^{n^2}$, where $a=(a_1, \cdots, a_{n^2}) \in F^{n^2}$ and $\gamma \in \bz^{n^2}$.  We recall that $T: F^{n^2} \to M_n(F)$ is an isomorphism, so $T^{-1}: M_n(F) \to F^{n^2}$ is also an isomorphism.
    
    We claim that for any two simple functions $h_1^{T^{-1}A^{-1},  T^{-1}(A^{-1}_1 t_2^{\Gamma_1})}, {g'_w}^{T^{-1}(A'_w) ,  T^{-1}( A'_w t_2^{\Gamma'_w})}: F^{n^2} \to \bc((X)),  \tau \in \GL(F^{n^2})$, $ \tau \in A_1(I_n+ t_2^{\Gamma_1}) A'_w(I_n+ t_2^{\Gamma'_w} )$  if the integral 
    $$\int_{F^{n^2}} h_1^{T^{-1}(A^{-1}),  T^{-1}(A^{-1}_1 t_2^{\Gamma_1})}(\tau^{-1}x)  \cdot{g'_w}^{T^{-1}(A'_w) ,  T^{-1} (A'_w t_2^{\Gamma'_w})}(x) dx
    $$
    
     is nonzero,  then it doesn't depends on $\tau$.

To prove the claim,	 we may assume $h_1= \Char_{K_1}, g'_w= \Char_{K_2}$ for some measurable  subsets  $K_1, K_2 \subset E^{n^2}$.  Let $V_1= (T^{-1}(A_1^{-1}) + T^{-1}(A^{-1}_1 )t_2^{T^{-1}(\Gamma_1)}  p^{-1}(K_1))$, $V_2= (T^{-1}(A'_w) + T^{-1} (A'_w)t_2^{T^{-1}(\Gamma'_w)}  p^{-1}(K_2))$, where $p: \co^{n^2} \to E^{n^2}$ is the projection map, then
	$$
  \int_{F^{n^2} }  \Char_{V_g}(\tau^{-1} x ) \Char_{V_h}( x)   dx =\int_{F^{n^2} }  \Char_{\tau V_g}(x ) \Char_{V_h}( x)dx.
	$$
Because we have $\tau =A_1(I_n + t_2^{\Gamma}M ) A'_w(I_n + t_2^{\Gamma'_w}N)$ for some $M, N \in M_n(\co)$.  By considering  the level of $t_2$, we have
	$$
T(	\tau V_1)  =  A_1(I_n + t_2^{\Gamma_1}M ) A'_w(I_n + t_2^{\Gamma'_w}N) A_1^{-1}(I_n + t_2^{\Gamma_1} T(p^{-1}K_1))=A'_w(I_n + t_2^{\Gamma_1}T(p^{-1}(K_1))) ,
	$$
	$$
 T(	V_2) =   A'_w (I_n+t_2^{\Gamma'_w}T(p^{-1}K_2)).
	$$
Then $\int_{F^{n^2} }  \Char_{\tau V_g}(x ) \Char_{V_h}( x)   dx =\mu(\tau V_g \cap V_h) = (\int_{E^2} \Char_{K_g} \Char_{K_h})X^{\sum \Upsilon_{ij}}$, where $\Upsilon \in M_n(\bz_+)$  only depends on $A'_w, \Gamma'_w$.
\end{proof}

\begin{theorem}
	 There is an $\bc((X))$-associative algebra structure on the space $\cl^M(\GL_n(F))$,  we denote the corresponding algebra by $\ch$, and call it the Hecke algebra of $\GL_n(F)$.
\end{theorem}

\begin{proof}
	We have proved that the convolution product is well-defined. The addition is $\bc((X))$-linear by definition. We only need to check that the convolution product is associative. For $f_1, f_2, f_3 \in \cl_M(G(F))$. We have 	$f_1  *(  f_2  * f_3) (y)=$
	\begin{eqnarray*}
        &&= \int_{\GL_n(F)} f_1(y g^{-1}) \left(  \int_{G(F)}  f_2(gh^{-1}) f_3(h) dh  \right) dg  \\
	                                  && =\int_{F^{n^2}}  f_1 \circ T_a (y g^{-1})  \left(  \int_{F^{n^2}}      f_2 \circ T_b(gh^{-1}) f_3\circ T_b(h) |\det T_b h|^{-n} dh  \right)  |\det T_a g|^{-n}dg    \\
	                                  && =\int_{F^{2n^2}}    f_1 \circ T_a (y g^{-1}) f_2 \circ  T_b(gh^{-1})  f_3  \circ T_b(h)  |\det T_bh|^{-n}  |\det T_ag|^{-n} dh  dg \\                    
	   \end{eqnarray*}
	   By the Fubini properties of our integral, we can reverse the order of integration.
	   \begin{eqnarray*}
	                                  && =\int_{F^{n^2}}    \left(   \int_{F^{n^2}}   f_1 \circ T_a(y g^{-1}) f_2 \circ  T_a(gh^{-1})   |\det T_a g|^{-n} dg  \right ) f_3 \circ T_b(h)  |\det T_bh|^{-n}    dh \\
	                                  && = \int_{F^{n^2}}  \left( \int_{F^{n^2 }}f_1 \circ T(y g^{-1}) f_2 \circ  T(gh^{-1})   |\det T g|^{-n} dg \right) f_3 \circ T(h)  |\det T h|^{-n} dh  \\
	                                    && = \int_{F^{n^2}}  \left( \int_{F^{n^2 }}f_1 \circ T(yh^{-1} (gh^{-1})^{-1}) f_2 \circ  T(gh^{-1})   |\det Tg|^{-n} d(gh^{-1}) \right) f_3 \circ T(h)  |\det T h|^{-n} dh  \\
	                                    && = \int_{F^{n^2}}  \left( (f_1 * f_2)  \circ T( yh^{-1})\right) f_3 \circ T(h)  |\det T h|^{-n} dh  \\
	                                    &&= (f_1 * f_2) * f_3(y).
   \end{eqnarray*}
\end{proof}
\section{Measurable representations of $\GL_n(F)$}

\subsection{Measurable representations of $\GL_n(F)$}

\begin{definition}
	A measurable $\bc((X))$-representation of $G(F)$ is a pair $(V, \pi)$, where $V$ is a $\bc((X))$-vector space, and 
	$$
	\pi: \GL_n(F) \to  \Aut_{\bc((X))}(V). 
	$$
	such that $\Stab_{v}$ is a measurable subgroup of $G$ for  every $v \in V$.
\end{definition}

Given a representation $\pi: G \to \GL(V)$,  for $f \in \ch $, one obtains a linear map
\begin{eqnarray*}
	\pi(f): &&V \to V   \\
	         && v  \mapsto  \int_{\GL_n(F)} f(g) \pi(g) \cdot v dg.
\end{eqnarray*}

 We obtain an action
$$
\ch  \times V \to V.
$$

\begin{lemma}
	The action $\ch  \times  V \to V$ is an algebra action.
\end{lemma}

\begin{proof}
	we have
	\begin{eqnarray*}
		\pi(f_1 * f_2) \cdot v & =& \int_{\GL_n(F)}    \left(\int_{\GL_n(F)}f_1(gh^{-1})f _2(h)  dh \right) \pi(g) \cdot v dg \\
		                               & =& \int_{\GL_n(F)}    \left(\int_{\GL_n(F)}f_1(g')f _2(h)  dh \right) \pi(g'h) \cdot v d (g'h) \\
		                                 &=& \int_{\GL_n(F)} \left(\int_{\GL_n(F)} f_1(g')  \pi(g') f_2(h)  \pi(h)  \cdot v   dh\right) dg 'h  \\
		                               & =& \int _{\GL_n(F)} f_1(g')  \pi(g')  \left(\int_{\GL_n(F)}  f_2(h)   \pi(h )  \cdot  v dh  \right)  dg'h   \\
		                                & =& \int _{\GL_n(F)} f_1(gh^{-1})  \pi(gh^{-1})  \left(\int_{\GL_n(F)}  f_2(h)   \pi(h )  \cdot  v dh  \right)  dg   \\
		                               & =&   \pi(f_1) (\pi(f_2)  \cdot v),        \\
	\end{eqnarray*}
	where the last equality comes from the translation property of the integral.
\end{proof}

\subsection{Representations on function spaces}
Suppose that we have a $\cl^{M}(\GL_n(F))$-module $V$, how do we get a measurable representation?   We first consider the action of $\GL_n(F)$ on $\cl^M(\GL_n(F))$ by
$$
(h \cdot f) (x) = f(h^{-1}x).
$$
This defines a representation of $\GL_n(F)$:
\begin{eqnarray*}
\GL_n(F) \times \cl^M(\GL_n(F))& \to&  \cl^M(\GL_n(F)) \\
                  (h, f)   & \mapsto&   h\cdot f.
\end{eqnarray*}

\begin{proposition}
 $\cl^M( \GL_n(F))$ is a measurable representation  of $\GL_n(F)$.
\end{proposition}

\begin{proof}
	By the translation invariant properties of the two-dimensional integral, we have $f(h^{-1}x) $ belongs to $\cl^M(\GL_n(F))$. To prove it is a measurable representation, we need to prove that 
$\mathrm{Stab}_f $ is  a measurable subgroup of $\GL_n(F)$.   Since $\cl^M(\GL_n(F))$ is generated by $g^{A, \Gamma}$, where $g$ are Bruhat-Schwartz functions on $\GL_n(E)$. We may suppose that $g = \Char^{\GL_n(E)}_K$ for a compact  subset $K$ of $\GL_n(E)$, and  $f = g^{A, \Gamma}$.  We then have
$$
g^{A, \Gamma} = \Char_V,   \quad   V=\left\{x \in \GL_n(F) |  x  \in  A(I_n+ t_2^{\Gamma} p^{-1}(K))  \right \},
$$
$$
g^{A, \Gamma} (h^{-1}x) = \Char_U,   \quad  U=\left\{x \in \GL_n(F) | h^{-1}x  \in  A(I_n+ t_2^{\Gamma} p^{-1}(K))  \right \}.
$$
By the assumption of $f$,
$$
\Stab_f = \left\{ h \in \GL_n(F)|  V=U \right\}.
$$
Since we have $\Gamma_{i,j} >0$, we get
$$
\Stab_f = \Stab_{A} \cap  \Stab_{p^{-1}K}.
$$
Since $K$ is a compact subset of $\GL_n(E)$,  for every $y \in K$, we have a compact open neighbourhood $K_yy \subset K$, where $K_y$ is a compact open subgroup of $\GL_n(E)$. Then 
$$
\{ K_yy| y \in K\}
$$ 
is an open cover of $K $, it has a finite  sub-cover $\{K_1y_1, \cdots,  K_sy_s \}$. Let
$$
H=  \cap_{1 \leq i \leq s} K_i.
$$
Then $K$ is left $H$ invariant. We claim that   $p^{-1}(K)$ is left $p^{-1}(H)$ invariant.

For  $w \in p^{-1}H$, and $z \in p^{-1}(K)$, we have
$$
p(wz)=p(w) p(z) \subset K,
$$
so $wz \subset p^{-1}(K)$.  On the other hand,  we have $\Stab_{A} =I_n$.  Thus  we have
\begin{enumerate}
	\item $\Stab_f=\{I_n\}$  if $A  \neq I_n$, and
	 \item  $\Stab_f =  I_n+ t_2^{\Gamma}p^{-1}(H)$ if $A= I_n$. 

\end{enumerate}
 We have $I_n$ is a measure $0$ subgroup, and 
\begin{eqnarray*}
	\int_{\GL_n(F)} \Char_{I_n+t_2^{\Gamma}p^{-1 }(H)}dx  &=& \int_{M_n (F)}  \Char_{I_n+t_2^{\Gamma}p^{-1}(H)} |\det x|^{-n}  dx  \\
	&=&\int_{M_n(F)} (\Char_H)^{0,0}(x) dx.  \\
 \end{eqnarray*}
We have $I_n+ t_2^{\Gamma}p^{-1}H$ is a measurable subset of $\GL_n(F)$,   clearly it is a subgroup of $\GL_n(F)$,  thus it is a measurable subgroup.
\end{proof}

\begin{lemma}
	Every element $f \in \cl^{H}(\GL_n(F))$ is in $\cl^{H}(\GL_n(F)//M)$ for some measurable subgroup  $M$ of $\GL_n(F)$. 
\end{lemma}
\begin{proof}
	We may suppose that $f= g^{0, \Gamma}$,  otherwise we can take $M=I_n$. Since $g$ is a Bruhat-Schwartz function on $\GL_n(E)$, we may suppose that $g= \Char^{\GL_n(E)}_K$ for a compact  open subset $K \subset \GL_n(E)$. Like the previous lemma,  for every $y \subset K$, we can choose a compact open subgroup $K_y \subset \GL_n(E)$, such that $yK_y \subset K$. We have $$
	\{ yK_y| y \in K\}
	$$ 
	is an open cover of $K $, it has a finite  sub-cover $\{y_1K_1, \cdots, y_s K_s \}$. Let
	$$
	H_r=  \cap_{1 \leq i \leq s} K_i.
	$$
	Then $K$ is right $H_r$-invariant. Similarly, we can find $H_l$ such that $K$ is left $H_l$ invariant. Let 
	$$
	H= H_r \cap H_l,
	$$
	we see that $K$ is bi-$H$-invariant.  Let $M=I_n+t_2^{\Gamma}p^{-1}(H)$, we claim that $f= g^{0, \Gamma}$ is bi-$M$-invariant. Suppose that $w_1= I_n+ t_2^{\Gamma}h_1, w_2=I_n+t_2^{\Gamma}h_2 \in I_n+ t_2^{\Gamma} p^{-1}(H), z=I_n+t_2^{\Gamma}k  \in I_0 +t_2^{\Gamma}p^{-1}(K)$, we then have
	$$
	w_1 z w_2=  (I_n+t_2^{\Gamma}h_2)(I_n+t_2^{\Gamma}h_2 )(I_n+t_2^{\Gamma}h_2)   \in I_n+   t_2^{ \Gamma} p^{-1} (K).
	$$
 We have
	$$
	g^{0, \Gamma} (x) = \Char_V,   \quad   V=\left\{x \in \GL_n(F) | x  \in  (I_n+ t_2^{\Gamma} p^{-1}(K))  \right \},
	$$
since $V$ is bi-M-invariant, we get $f \in \cl^M(\GL_n(F)//M)$.
\end{proof}

\begin{remark}
	Since $\cl^M(\GL_n(F))$ is generated by $ g^{A, \Gamma}$   for  $g$  compactly support functions  (hence locally constant functions). We then get for  a $f \in \cl^M(GL_n(F)//M)$, we have
	$$
	f  =\sum  f(g_i) \Id_{Mg_iM}.
	$$
	We notice that usually this sum has infinite items.
\end{remark}

\bibliographystyle{alpha}
\bibliography{main}

\end{document}